\documentclass[11pt]{amsart}

\usepackage{geometry}
\usepackage[T1]{fontenc}
\usepackage[utf8]{inputenc}
\usepackage{graphicx}
\usepackage{amssymb}
\usepackage{epstopdf}
\usepackage{tikz}

\usepackage[section]{placeins}
\usepackage{extpfeil}

\usepackage{todonotes}%%% this is to write comments%%%%%%

\usepackage{enumitem,amssymb}
\newlist{todolist}{itemize}{2}
\setlist[todolist]{label=$\square$}
\usepackage{pifont}

\usepackage{multicol}
\usepackage{amsfonts}
\usepackage{amsmath}
\usepackage{amsthm}
\usepackage{euscript}
\usepackage{amssymb}
\usepackage[colorlinks]{hyperref}
\usepackage{mathrsfs}
\usepackage{xypic}
\usepackage{tikz}
\usepackage{caption}
\usepackage{subfig}
\usepackage{float}
\usepackage{mathtools}
\usepackage{rotating}

\xyoption{all}

\usepackage{fancyvrb}
\RecustomVerbatimEnvironment{Verbatim}{Verbatim}{formatcom=\footnotesize,baselinestretch=0.6,xleftmargin=5mm,xrightmargin=5mm}

\usepackage[ruled]{algorithm}
\usepackage{algorithmic}

\theoremstyle{plain}
\newtheorem{mainthm}{Theorem}
\newtheorem{thm}{Theorem}[section]

\newtheorem{lem}[thm]{Lemma}
\newtheorem{prop}[thm]{Proposition}

\theoremstyle{definition}

\newtheorem{rem}[thm]{Remark}

\newcommand{\id}{\mathrm{id}}

\newcommand{\gr}{\mathrm{gr}}
\newcommand{\fg}{\mathfrak{g}}
\newcommand{\fsl}{\mathfrak{sl}}

\newcommand{\Z}{\mathbb Z}
\newcommand{\C}{\mathbb C}
\newcommand{\N}{\mathbb N}

\newcommand{\PP}{{\mathbb P}}

\newcommand{\QQ}{{\mathbb Q}}

\newcommand{\OO}{\mathcal{O}}
\newcommand{\FF}{\mathcal{F}}

\newcommand{\EE}{\mathcal{E}}
\newcommand{\W}{\mathcal{W}}
\newcommand{\cZ}{\mathcal{Z}}
\newcommand{\NN}{\mathcal{N}}
\newcommand{\V}{\mathcal{V}}
\newcommand{\LL}{\mathcal{L}}
\newcommand{\GGG}{\mathcal{G}}
\newcommand{\KK}{\mathcal{K}}
\newcommand{\Q}{\mathcal{Q}}
\newcommand{\A}{\mathcal{A}}
\newcommand{\B}{\mathcal{B}}

\newcommand{\op}{\oplus}
\newcommand{\ot}{\otimes}

\renewcommand{\ge}{\geqslant}
\renewcommand{\le}{\leqslant}

\DeclareMathOperator{\HH}{H}

\DeclareMathOperator{\Hom}{Hom}
\DeclareMathOperator{\im}{Im} 
\DeclareMathOperator{\rk}{rk}
\DeclareMathOperator{\Ker}{Ker} 
\DeclareMathOperator{\coker}{Coker}

\DeclareMathOperator{\SL}{\mathrm{SL}}

\begin{document}

\sloppy

\title[]{A construction of equivariant bundles\\ on the space of symmetric forms}

\author{Ada Boralevi}
\address{Dipartimento di Scienze Matematiche \lq\lq G. L. Lagrange\rq\rq, (Dipartimento di Eccellenza 2018--2022), Politecnico di Torino, Corso Duca degli Abruzzi 24, 10129 Torino, Italy}
\email{\href{mailto:ada.boralevi@polito.it}{ada.boralevi@polito.it}}

\author{Daniele Faenzi}
\address{Institut de Math\'ematiques de Bourgogne,
UMR CNRS 5584,
Universit\'e de Bourgogne et Franche Comt\'e,
9 Avenue Alain Savary,
BP 47870,
21078 Dijon Cedex,
France}
\email{\href{mailto:daniele.faenzi@u-bourgogne.fr}{daniele.faenzi@u-bourgogne.fr}}

\author{Paolo Lella}
\address{Dipartimento di Matematica, Politecnico di Milano, Via Bonardi 9,  20133 Milano, Italy}
\curraddr{}
\email{\href{mailto:paolo.lella@polimi.it}{paolo.lella@polimi.it}}
\urladdr{\url{http://www.paololella.it/}}

\thanks{The first named author is supported by ``Politecnico di Torino - Starting grant per giovani ricercatori a TD A e B''. All authors are members of GNSAGA}



\begin{abstract}
We construct stable vector bundles on the space $\PP(S^d \C^{n+1})$
of symmetric forms of degree $d$ in $n+1$ variables which are
equivariant for the action of $\SL_{n+1}(\C)$, and admit an equivariant
free resolution of length $2$.
For $n=1$, we obtain new examples of stable vector bundles of rank $d-1$ on
$\PP^d$, which are moreover equivariant for $\SL_2(\C)$. The
presentation matrix of these bundles attains
Westwick's upper bound for the dimension of vector spaces of matrices
of constant rank and fixed size.
\end{abstract}

\maketitle

%%%%%%%%%%%%%%%%%%%%%%%%%%%%%%%%%%%%%%%%%%%%%%%%%%%%%%%%%
\section*{Introduction}
%%%%%%%%%%%%%%%%%%%%%%%%%%%%%%%%%%%%%%%%%%%%%%%%%%%%%%%%%

It is notoriously difficult to construct rank-$r$ non-splitting vector bundles (i.e.~not
isomorphic to the direct sum of line bundles) on
$\PP^N$ if $r$ is small with respect to $N$. A famous conjecture of
Hartshorne entails that the task is in fact impossible when $r < N/3$.
As for the meaning of the word \textit{small} here, it basically
refers to any value of $r \le N-1$ (say for $N \ge 4$). For instance, no example of non-splitting vector bundle of rank $r \le N-3$ is
known, at least if we work in characteristic zero, which we tacitly
assume from now on.
Moreover, only two sporadic (yet quite important)
constructions are available for $r = N-2$, that are due to
Horrocks-Mumford (for $N=4$) and Horrocks (for $N=5$). Once such
a vector bundle is constructed, one can pull it back via any
finite self-map of $\PP^N$ to obtain a new bundle. Together with the
method of affine pull-backs developed by Horrocks
(cf.~\cite{horrocks:rank-three}, see also \cite{ancona-ottaviani:horrocks}), this
essentially exhausts the set of techniques currently available, to the best of our knowledge.

The situation improves slightly when $r=N-1$. In this range, basically two
classes of bundles are known: instanton bundles and Tango
bundles (we refer to \cite{OSS} for a treatment of these classical cases). These have been
recently generalized by Bahtiti,
cf.~\cite{bahtiti-correlation,bahtiti:tango,bahtiti:2n+1}. More
examples are given by Cascini's weighted Tango bundles,
see \cite{cascini:tango}, and by the Sasakura bundle of rank $3$ on $\PP^4$,
cf.~\cite{anghel:sasakura}.
That is roughly the list of all the examples known so far in this realm.

From another perspective, one may try to construct vector bundles
starting from their presentation matrix. Such matrix will have
constant corank $r$ when evaluated at any point of $\PP^N$. While the
opposite procedure (constructing a matrix from a bundle) is also interesting, as we tried to show in
\cite{bo_fa_me,adp}, it is not quite clear how to construct matrices
of constant corank $r \le N-1$, especially so if we impose constraints
on the matrix, for instance that it should have a given size, say $a
\times b$, or that
its coefficients should have a fixed degree, notably degree one.
A simple calculation (cf.~\cite{Westwick1}) implies
(say $a \le b$) that such a matrix
can exist a priori only if $r+1$ divides $(a-1)!/(b-r-1)!$. Under such
divisibility condition, some matrices of linear forms of size
$(b-r+1)\times(b-r+n-1)$ and constant corank $n-1$ where given in
\cite{Westwick_5}, attaining the upper bound 
for the dimension of vector spaces of matrices
of constant rank and fixed size. The construction
is a bit obscure to us, and in any case it says very little about the bundle itself.

The goal of this paper is to introduce a simple technique to
construct non-splitting vector bundles on $\PP^N$. For this, one has to view
$\PP^N$ as the space of homogeneous forms of degree $d$ on $\PP^n$ for
some $(d,n)$, and use a little representation theory of
$G=\SL_{n+1}(\C)$. 

The resulting bundles satisfy a much stronger property than just being
non-splitting, namely they are stable in sense of Mumford-Takemoto, or slope-stable.
Also, by construction they are homogeneous for the action of $G$ and again by definition their dual bundles are presented by a matrix of linear forms which is
equivariant for $G$.

For $n=1$, one has $N=d$ and our bundles have rank $d-1$. As we will
see, for $d \ge 4$ these bundles turn out to
be different from all the bundles of rank $d-1$ on $\PP^d$ constructed
so far (except for the single case of the classical Tango bundle).

Also, the matrix of linear
forms will have size $(b-r+1)\times(b-r+n-1)$ and constant corank
$n-1$, thus giving a new approach to achieve Westwick's bound.

Finally, for $n=1$ and $d=3$, our bundles agree with the
$\SL_2(\C)$-invariant instantons defined and studied in
\cite{dani_ist_omog}. These instantons are parametrized by $\N$ in
the sense that the second Chern class (the so-called the ``charge'')
of an $\SL_2(\C)$-invariant instanton over $\PP^3=\PP(V_3)$ must equal
${m\choose 2}$ for
some integer $m \ge 2$, and given such $m$ there is one and only
one such instanton. This instanton is precisely $\W_{m,3}$.
However our results generalize the construction 
and simplify some of the proofs given in that paper.

For higher $n$, our bundles have rank bigger than the dimension
$N=\binom{d+n}{n}-1$ of the
ambient space. Nevertheless, they seem interesting as are they stable, homogeneous
for the action of a rather big group operating on $\PP^N$, but still
of much smaller rank than most $\SL_{N+1}(\C)$-homogeneous
bundles. To study them we will pull back to $\PP^n$ via the Veronese map and use the theory of
$\SL_{n+1}(\C)$-homogeneous bundles in terms of quiver representations developed in \cite{OR}.

\medskip

We now state our results more precisely.
For an integer $n \ge 1$, let $V$ be a complex vector space of
dimension $n+1$ and let $G \simeq \SL_{n+1}(\C)$ denote
the general linear group of automorphisms of $V$.
The representation theory of $G$ is governed by the 
fundamental 
weights $\lambda_1,\ldots,\lambda_{n}$ of $G$, in the sense that an irreducible representation of $G$ is uniquely determined by
its dominant weight $\lambda = a_1\lambda_1 + \cdots + a_n \lambda_n$,
where $a_i \in \N$ for all $i$. We write $V_\lambda$ for this
representation. By convention, the standard representation is  $V_{\lambda_1}$
and we often write $V=V_{\lambda_1}$.
We also abbreviate $V_d$ for $V_{d \lambda_1}$.

Now suppose $n \ge 2$, and take integers $d \le e$.  The Littlewood-Richardson rule gives:
\begin{equation}\label{tensor prod dec n}
V_d \ot V_e \simeq \bigoplus_{i=0}^{d} V_{(d+e-2i)\lambda_1 + i \lambda_2}.
\end{equation}
From \eqref{tensor prod dec n} we extract a $G$-invariant non-zero morphism:
\[
V_d \ot V_e \to V_{(d+e-2)\lambda_1 + \lambda_2}.
\]
Assume that  $e$ is a multiple of $d$, say $e=(m-1)d$; the inclusion of the second summand $V_{(md-2)\lambda_1 + \lambda_2}$ 
from the decomposition \eqref{tensor prod dec n} into the product $V_d \ot V_{(m-1)d}$ induces a $G$-equivariant morphism
\[
\Phi_{m,d}: V_{(md-2)\lambda_1+\lambda_2} \ot \OO_{\PP V_d} \longrightarrow V_{(m-1)d} \ot \OO_{\PP V_d}(1),
\]
which is a matrix of linear forms. We put
$\W_{m,d}:=\Ker(\Phi_{m,d})$. 

\medskip

The sheaves $\W_{m,d}$ constitute the main object of study of this paper. Here is our main result about them.

\begin{mainthm}\label{risultato generale}
Let $d \ge 2$ and $m \ge 2$ be integers. Then
$\Phi_{m,d}$ has constant corank $1$ and $\W_{m,d}$  is a slope-stable $\SL_{n+1}(\C)$-equivariant vector bundle on $\PP V_d$ of rank
\[
\rk(\W_{m,d})=(md-1)\tbinom{md+n-1}{n-1} - \tbinom{(m-1)d+n}{n} + 1
\]
that fits into:
\begin{equation}\label{ex seq}
0 \to  \W_{m,d} \to  V_{(md-2)\lambda_1+\lambda_2} \ot \OO_{\PP V_d} \xrightarrow{\Phi_{m,d}} V_{(m-1)d} \ot \OO_{\PP V_d}(1) \xrightarrow{\Psi_{m,d}} \OO_{\PP V_d}(m) \to 0,
\end{equation}
\end{mainthm}

\medskip

If $n=1$ (so that $V \simeq \C^2$) formula \eqref{tensor prod dec n} simplifies significantly and reads:
\begin{equation}\label{tensor prod dec 2}
V_d \ot V_{(m-1)d} =V_{md} \op V_{md-2} \op \ldots \op V_{(m-2)d}.
\end{equation}
The second summand is just the symmetric power $V_{md-2}$, and the space of symmetric forms $\PP V_d \simeq \PP^d$. 
The following analogue of Theorem \ref{risultato generale} is our
second result.

\begin{mainthm}\label{risultato SL2}
Let $d \ge 2$ and $m \ge 2$ be integers. Then $\W_{m,d}$ is a slope-stable  vector
bundle of rank $d-1$ on $\PP^d$, homogeneous under the action of
$\SL_2(\C)$, fitting into:
\begin{equation}\label{ex seq sl2} 
0 \rightarrow \W_{m,d} \rightarrow  V_{md-2} \ot \OO_{\PP^d} \xrightarrow{\Phi_{m,d}} V_{(m-1)d} \ot \OO_{\PP^d}(1) \xrightarrow{\Psi_{m,d}} \OO_{\PP^d}(m) \rightarrow 0.
\end{equation}
\end{mainthm}

\medskip

The bundles we construct here, as well as the
presentation matrices defining them, are actually
defined over $\PP^n_{\QQ}$ and therefore over $\Z$. However one
cannot reduce modulo an arbitrary prime number $p$ to obtain bundles
defined in characteristic $p$ (unless $p$ is
high enough), as the rank of the defining matrices may drop modulo $p$.

The presentation matrices can be defined in an algorithmic fashion
simply by using the action of the Lie algebra of $\SL_{n+1}(\C)$. We
provide an ancillary file containing a {\tt Macaulay2} package to do this on a computer.

\bigskip

The paper is structured as follows. In \S \ref{equivariant complex} we
prove that the maps appearing in the sequence defining $\W_{m,d}$ 
compose to zero, and that the resulting equivariant complex is exact
at the sides. In \S \ref{SL2} we prove our main result for $n=1$,
i.e.~for the case of $\SL_2(\C)$-bundles. This is intended to guide
the reader through the argument, which is a bit different (and much
simpler) in this case. Also, this allows to quickly exhibit our $\SL_2(\C)$-bundles, which are 
the most interesting ones as far as the search of low-rank bundles on
$\PP^d$ is concerned. In \S \ref{G} we treat the case of higher $n$, where the
treatment of homogeneous bundles via representations of quivers comes into play.
In \S \ref{varie} we show that our bundles are always new except for
the case of the classical Tango bundle.
\bigskip

We would like to thank L. Gruson and G. Ottaviani for fruitful comments.

%%%%%%%%%%%%%%%%%%%%%%%%%%%%%%%%%%%%%%%%%%%%%%%%%%%%%%%%%
\section{The equivariant complex} \label{equivariant complex}
%%%%%%%%%%%%%%%%%%%%%%%%%%%%%%%%%%%%%%%%%%%%%%%%%%%%%%%%%

Recall the fundamental sequence appearing in Theorem \ref{risultato generale}:
\begin{equation}
0 \to  \W_{m,d} \to  V_{(md-2)\lambda_1+\lambda_2} \ot \OO_{\PP V_d} \xrightarrow{\Phi_{m,d}} V_{(m-1)d} \ot \OO_{\PP V_d}(1) \xrightarrow{\Psi_{m,d}} \OO_{\PP V_d}(m) \to 0, \tag{\ref{ex seq}}
\end{equation}
and its analogue for the case $n=1$, from Theorem \ref{risultato SL2}:
\begin{equation}
0 \rightarrow \W_{m,d} \rightarrow  V_{md-2} \ot \OO_{\PP^d} \xrightarrow{\Phi_{m,d}} V_{(m-1)d} \ot \OO_{\PP^d}(1) \xrightarrow{\Psi_{m,d}} \OO_{\PP^d}(m) \rightarrow 0. \tag{\ref{ex seq sl2}}
\end{equation}
In this section we show that these maps form  equivariant complexes of vector bundles.

\medskip

\begin{lem}\label{composizione nulla}
For $n \ge 2$, the space of $G$-invariant morphisms $\Hom(V_{(md-2)\lambda_1+\lambda_2}, S^mV_d)^G$ is zero. In particular, the same is true for the composition of the two maps 
\[V_{(md-2)\lambda_1+\lambda_2} \ot \OO_{\PP V_d} \xrightarrow{\Phi_{m,d}} V_{(m-1)d} \ot \OO_{\PP V_d}(1) \xrightarrow{\Psi_{m,d}}  \OO_{\PP V_d}(m).
\]  
The same result holds for $n=1$, the space $\Hom(V_{md-2}, S^mV_d)^G$, and the composition 
\[V_{md-2} \ot \OO_{\PP^d} \xrightarrow{\Phi_{m,d}} V_{(m-1)d} \ot \OO_{\PP^d}(1) \xrightarrow{\Psi_{m,d}}  
\OO_{\PP^d}(m).
\]  
\end{lem}

\begin{proof} In virtue of Schur's lemma, it is enough to show that the irreducible representation $V_{(md-2)\lambda_1+\lambda_2}$ does not appear in the decomposition of $S^mV_d$.

We follow the standard notation from \cite{FH} and denote by $H_i$ the
diagonal matrix $E_{i,i}$, and by $L_i$ the linear operator such that
$L_i(H_j)=\delta_{ij}$, so that the fundamental weights of $G$ are
$\lambda_i=L_1 + \ldots + L_i$, for $i=1,\ldots,n$. The Lie algebra
$\fg$ of $G$ is generated by $E_{i,j}$, $E_{j,i}$, and $H_i-H_j$ with the standard commutation relations $[H_i-H_j,E_{i,j}]=2E_{i,j}$, $[H_i-H_j,E_{j,i}]=-2E_{j,i}$, $[E_{i,j},E_{j,i}]=H_i-H_j$, again for $1\le i <j \le n$.

Now suppose that $V$ is generated by $\mathsf{x_1},\ldots,\mathsf{x_{n+1}}$; then the space of symmetric $d$-forms $V_d$ has a basis defined by $x^{(k_1,\ldots,k_{n+1})}=\mathsf{x_1}^{k_1}\mathsf{x_2}^{k_2}\cdots \mathsf{x_{n+1}}^{k_{n+1}}$ and indexed by all partitions $(k_1,\ldots,k_{n+1})$ of $d$. We rename the basis elements of $V_d$ as $y_1,\ldots,y_\alpha$, with $\alpha=\binom{n+d}{d}$, according to the lexicographic ordering on the monomials of degree $d$ in the variables $\mathsf{x_1},\ldots,\mathsf{x_{n+1}}$.
We repeat the process for the symmetric power $V_{(m-1)d}$, endowing it with the basis $z_1,\ldots,z_\beta$, with $\beta=\binom{n+(m-1)d}{(m-1)d}$.

The natural action of $G$ extends linearly to the product $V_d \otimes V_{(m-1)d}$, which in turn splits as in \eqref{tensor prod dec n}. The highest weight vector of the irreducible representation $V_{(md-2)\lambda_1+\lambda_2}$ of highest weight $(md-2)\lambda_1+\lambda_2= (md-1)L_1+L_2$ is $y_1z_2-y_2z_1$.

On the other hand, $G$ also acts on the symmetric power $S^mV_d$; with our notation, $S^mV_d$ has a basis defined by monomials of type $y_1^{h_1}y_2^{h_2}\cdots y_\alpha^{h_\alpha}$, indexed by all partitions $(h_1,\ldots,h_\alpha)$ of $m$. The only such monomial with highest weight $(md-1)L_1+L_2$ is $y_1^{m-1}y_2$, and this makes it (up to a constant) the only candidate for being the highest weight vector of $V_{(md-2)\lambda_1+\lambda_2}$. Notice however that:
\[
E_{12}(y_1^{m-1}y_2)=E_{12}(\mathsf{x_1}^{(m-1)d}\mathsf{x_1}^{d-1}\mathsf{x_2})=\mathsf{x_1}^{md}=y_1^m \neq 0,
\]
hence $(md-1)L_1+L_2$ cannot be a highest weight in $S^mV_d$, and $V_{(md-2)\lambda_1+\lambda_2}$ cannot appear as irreducible summand in its decomposition.

\smallskip

For the second part of the statement, notice that
\[
\Psi_{m,d} \circ \Phi_{m,d} \in \Hom(V_{(md-2)\lambda_1+\lambda_2} \ot \OO_{\PP V_d},  \OO_{\PP V_d}(m))^G \simeq \Hom(V_{(md-2)\lambda_1+\lambda_2}, S^mV_d)^G.
\]
Finally, in the case of $\SL_2(\C)$, recall that the Lie algebra $\fg$ of $G$ is
generated by $X$, $Y$, and $H$, with $[H,X]=2X$, $[H,Y]=-2Y$, and
$[X,Y]=H$; the same proof as above applies with minor modifications
which we omit.
\end{proof}

\medskip

In the study of the sequences \eqref{ex seq} and \eqref{ex seq sl2},
it is useful to restrict them to the closed orbit of the $G$-action on $\PP V_d$, namely the Veronese variety. For this, let
$v_{n,d}$ be the Veronese map $\PP V \to \PP V_d$ given by the
complete linear system $|\OO_{\PP V}(d)|$, and put 
\[
X_d:=\im(v_{n,d} : \PP V \to \PP V_d).
\]

 The idea behind this is that  $\PP V$ is a
 $G$-homogeneous space and the pull-back of $\W_{n,d}$ to $\PP V$ 
 is a $G$-homogeneous bundle.
 This remark enables to prove several results concerning this pull-back.
 The next observation is that many of our key statements extend to the
 whole ambient space $\PP V_d$ by continuity and $G$-equivariance,
 because the $G$-orbit of any point of $\PP V_d$ contains $X_d$ in its
 Zariski closure.

\bigskip

Restricting \eqref{ex seq} and \eqref{ex seq sl2} to $X_d$ and pulling
back via $\PP V \simeq \PP^n \to X_d$ we get the sequences: 
\begin{equation}\label{restr ex seq}
0 \to  \V_{m,d} \to  V_{(md-2)\lambda_1+\lambda_2} \ot \OO_{\PP^n} \xrightarrow{\phi_{m,d}} V_{(m-1)d} \ot \OO_{\PP^n}(d) 
\xrightarrow{\psi_{m,d}} \OO_{\PP^n}(md) \to 0,
\end{equation}
and 
\begin{equation}\label{restr ex seq sl2}
0 \to  \V_{m,d} \to  V_{md-2} \ot \OO_{\PP^1} \xrightarrow{\phi_{m,d}} V_{(m-1)d} \ot \OO_{\PP^1}(d) 
\xrightarrow{\psi_{m,d}} \OO_{\PP^1}(md) \to 0,
\end{equation}
respectively, where in both cases $\V_{m,d} \simeq \W_{n,d}|_{X_d}$ is a vector bundle on $\PP^n$ which is homogeneous under $G$, the map $\phi_{m,d}$ is the pull-back to $\PP^n$ of $\Phi_{m,d}$, and similarly $\psi_{m,d}$ is the pull-back of $\Psi_{m,d}$.

\medskip

Studying these restricted sequences we obtain the following:

\begin{lem}\label{psi surj} 
The morphism $\Psi_{m,d}: V_{(m-1)d} \ot \OO_{\PP V_d}(1) \longrightarrow  \OO_{\PP V_d}(m)$ is surjective.
\end{lem}

\begin{proof} 
Let us consider a different Veronese embedding than before, namely the
one of degree $(m-1)d$, $v_{n,(m-1)d}$.
On the projective space $ \PP V_{(m-1)d}$ we have the Euler sequence
with the natural surjection 
\[
V_{(m-1)d} \ot \OO_{\PP V_{(m-1)d}} \rightarrow \OO_{\PP V_{(m-1)d}}(1).
\]
This surjection, pulled-back via $v_{n,(m-1)d}$ to $\PP^n$, becomes the
obvious epimorphism:
\[V_{(m-1)d} \ot \OO_{\PP^n} \rightarrow \OO_{\PP^n}((m-1)d).\]
This map is precisely the morphism $\psi_{m,d}$, twisted by $\OO_{\PP
  V}(-d)$, so $\psi_{m,d}$ is surjective.
The surjectivity of $\Psi_{m,d}$ follows, because the rank of the map can only increase with respect to the value on the closed orbit.
\end{proof}

\begin{rem}
The restricted morphism $\psi_{m,d}$ also corresponds to the natural
surjection $\mathcal{P}^k(F) \twoheadrightarrow F$ of the \emph{bundle
  of $k$-jets} (also known as  \emph{principal parts sheaf}) of a
vector bundle $F$ onto $F$ itself. Indeed in \cite{perkinson} it is
shown that the $k$-jets of line bundles on a projective space $\PP V$
have the simple form $\mathcal{P}^k( \OO_{\PP V}(h))= V_k \ot
\OO_{\PP V}(h-k)$. In particular,  
$\mathcal{P}^{(m-1)d}( \OO_{\PP V}(md))= V_{(m-1)d} \ot \OO_{\PP V}(d)$. 
\end{rem}

\smallskip

Lemmas \ref{composizione nulla} and \ref{psi surj} entail that both
sequences \eqref{ex seq} and \eqref{ex seq sl2} are equivariant
complexes, exact at the sides. In the next two sections we will show that
exactness holds in the middle, as well, and this will conclude the proof
of our main results.

%%%%%%%%%%%%%%%%%%%%%%%%%%%%%%%%%%%%%%%%%%%%%%%%%%%%%%%%%
\section{The case of binary forms} \label{SL2}
%%%%%%%%%%%%%%%%%%%%%%%%%%%%%%%%%%%%%%%%%%%%%%%%%%%%%%%%%

This section is devoted to the special case $G=\SL_2(\C)$. We start 
by proving our main theorem in this case. In \S \ref{varie} we will draw a
few more remarks about our bundles of rank $d-1$ on $\PP^d$, which we
see as the space of binary forms of degree $d$.

\subsection{The equivariant matrix of linear forms}
Recall from the proof of Lemma \ref{composizione nulla} that the Lie
algebra $\fsl_2(\C)$ is generated by $X$, $Y$, and $H$, with $[H,X]=2X$,
$[H,Y]=-2Y$, and $[X,Y]=H$; moreover, we consider $V$ as being
generated by $\mathsf{x_1}$ and $\mathsf{x_2}$, so that  $V_{d}$ has a
basis defined by $y_k=\mathsf{x_1}^{d-k+1}\mathsf{x_2}^{k-1}$, for
$k=1,\ldots,d+1$. Similarly  $V_{(m-1)d}$ has a basis defined by
$z_h=\mathsf{x_1}^{(m-1)d-h+1}\mathsf{x_2}^{h-1}$, for
$h=1,\ldots,(m-1)d+1$. 

An element of weight $md-2j$ in $V_d \otimes V_{(m-1)d}$ is a linear combination
\[
v^{(j)} = \sum_{i=1}^{j+1} c_{i,j}\, y_{j-i+2} z_{i}
\]
and acting with $Y$ on $v^{(j-1)}$, one obtains
\[
\begin{split}
Y\big(v^{(j-1)}\big) = {}& (d-j+1) c_{1,j-1}\, y_{j+1} z_1 + {} \\
& \sum_{i=2}^{j}  \Big[ (d-j+i) c_{i,j-1} + \big((m-1)d-i+2\big) c_{i-1,j-1} \Big] y_{j-i+2} z_i + {}\\
& \big((m-1)d-j+1\big) c_{j,j-1} \, y_1 z_{j+1}.
\end{split}
\]
As the highest weight vector of $V_{md-2}$ is $y_1z_2-y_2z_1$, a basis of the representation $V_{md-2}$ in $V_d \otimes V_{(m-1)d}$ is given by the set $\{v^{(1)},\ldots,$ $v^{(md-1)}\}$ with coefficients $c_{i,j}$ defined by $c_{1,1} = -1$, $c_{2,1} = 1$ and for $2 \leqslant j \leqslant md-1$
\begin{equation}\label{eq:ricorrenza}
c_{i,j} = \begin{cases}
(d-j+1)c_{1,j-1}& i = 1,\\
(d-j+i) c_{i,j-1} + \big((m-1)d-i+2\big) c_{i-1,j-1}& 2 \leqslant i \leqslant j,\\
\big((m-1)d-j+1\big) c_{j,j-1} & i = j+1.
\end{cases}
\end{equation}
Resolving the recurrence relation, one finds
\[
c_{i,j} = \begin{cases}
- \smashoperator[r]{\prod\limits_{k=2}^j} (d-k+1),& i=1,\\
- \left[ \smashoperator[r]{\prod\limits_{k=1}^{i-2}}\big( (m-1)d-k\big) \right]\left[ \smashoperator[r]{\prod\limits_{k=1}^{j-i+1}} (d-k+1)\right] \left[\binom{j-1}{i-1}m - \binom{j}{i-1}\right], & i \geqslant 2.
\end{cases}
\]
These coefficients can be simplified when describing the matrix $\Phi_{m,d}$. In fact, given $j \geqslant d+2$, the coefficient $c_{i,j}$ vanishes if $i \leqslant j-d$ or $i \geqslant j+2$ and for $j-d+1 \leqslant i \leqslant j+1$ all coefficients $c_{i,j}$ are multiple of the product $\prod_{k=1}^{j-d-1} \big( (m-1)d-k\big)$ that can be removed. Moreover, a basis of the representation $V_{md-2}$ in $V_d \otimes V_{(m-1)d}$ can also be obtained starting from the lowest weight vector $y_{d}z_{(m-1)d+1} - y_{d+1}z_{(m-1)d} $ and acting with $X$. In this way, one obtains a recurrence relation analogous to \eqref{eq:ricorrenza} that reveals a symmetry of the coefficients appearing in the matrix $\Phi_{m,d}$. Taking into account these remarks, the matrix representing $\Phi_{m,d}$ in the case of $G=\SL_2(\C)$ is
\[
\big(\Phi_{m,d}\big)_{i,j} = \begin{cases}
\widetilde{c}_{i,j}\, y_{j-i+2}, & 1 \leqslant j-i+2 \leqslant d+1, \\
0, & \text{otherwise},
\end{cases}
\]
where
\[
\widetilde{c}_{i,j} = \begin{cases}
 - \smashoperator[r]{\prod\limits_{k=2}^j} (d-k+1) & \begin{array}{c} j \leqslant \left\lceil \frac{md-1}{2} \right\rceil\\\text{and }  i = 1 \end{array},\\
 - \left[ \smashoperator[r]{\prod_{k=\max(1,j-d)}^{i-2}}\big( (m-1)d-k\big) \right]\left[ \smashoperator[r]{\prod\limits_{k=1}^{j-i+1}} (d-k+1)\right] \left[\binom{j-1}{i-1}m - \binom{j}{i-1}\right] & \begin{array}{c} j \leqslant \left\lceil \frac{md-1}{2} \right\rceil\\ \text{and } i \geqslant 2 \end{array},\\
-\widetilde{c}_{(m-1)d-i+2, md - j}  &  \begin{array}{c}  j > \left\lceil \frac{md-1}{2} \right\rceil \end{array}.
\end{cases}
\]

\begin{figure}
\begin{multicols}{2}

\subfloat[][The matrix $\Phi_{4,4}$.]
{
\begin{sideways}
\begin{minipage}{0.85\textheight}
\input{example.tex}
\end{minipage}
\end{sideways}
}

\columnbreak

\subfloat[][The matrix $\Phi_{3,5}$.]
{
\begin{sideways}
\begin{minipage}{0.85\textheight}
\input{example2.tex}
\end{minipage}
\end{sideways}
}
\end{multicols}
\caption{Two examples of the matrix $\Phi_{m,d}$ in the case of $G = \SL_2(\C)$.}
\end{figure}

\subsection{Proof of Theorem \ref{risultato SL2}}

We proceed in several steps, articulated along the next subsections,
which we briefly outline here. First, we show exactness of the
equivariant complex. Then, we pull-back to $\PP^1$ via the Veronese
embedding of $\PP^1$ in $\PP^d$ determined by the identification
$\PP^d = \PP V_d$. We study the pull-back bundle $\V_{m,d}$ of $\W_{m,d}$ and
prove that it is isomorphic to $V_{d-2}\otimes \OO_{\PP^1}((1-m)d)$.
We finally deduce the stability of $\W_{m,d}$.

\subsubsection{Exactness of the equivariant complex}

As mentioned above, Lemmas \ref{composizione nulla} and \ref{psi surj} entail 
that sequence \eqref{ex seq sl2} is a complex, exact at the sides; 
in particular, $\Phi_{m,d}$ is a $((m-1)d+1)\times(md-1)$ matrix of linear forms in $d+1$ variables, and whose rank is at most $(m-1)d$. 

On the other hand, $\rk(\Phi_{m,d})$ is bounded below by the rank of
$\Phi_{m,d}\vert_{y}$, where $y$ is any point of the
closed orbit, the rational normal curve $X_d$ of degree $d$ in
$\PP^d$. At the point $y=[1:0:\ldots:0]  \in X_d$, the entries of the matrix $\Phi_{m,d}\vert_{y}$ are all zero except that on the subdiagonal, where the values are
\[
\big(\Phi_{m,d}\vert_{y} \big)_{j+1,j} = \prod_{k=1}^{j-1} \big((m-1)d-k\big)\neq 0,\qquad \forall\ j=1,\ldots,(m-1)d.	
\]
Hence the rank is $(m-1)d$, which is what we wanted. 

From \eqref{ex seq sl2} we compute that the vector bundle $\W_{m,d}$
has rank equal to $md-1-(m-1)d=d-1$, as required.

\subsubsection{The pulled-back image bundle}

Recall the  morphism of bundles $\psi_{m,d}$, defined on
$\PP^1$ as pull-back of $\Psi_{m,d}$ via the Veronese map $v_{1,d} :
\PP^1 \to \PP^d$. We first study the image of $\psi_{m,d}$, so
put $\LL_{m,d} = \im(\psi_{m,d})$. This is a vector bundle defined by
the exact sequence: 
\[
0 \to \LL_{m,d} \to V_{(m-1)d} \ot \OO_{\PP^1}(d) \xrightarrow{\psi_{m,d}} \OO_{\PP^1}(md) \to 0.  
\]
We first show that:
\begin{equation}
  \label{LP1}
\LL_{m,d} \simeq V_{(m-1)d-1} \otimes \OO_{\PP^1}(d-1).  
\end{equation}
To see this, note that for each $t \ge 0$ the map $\psi_{m,d}$ induces an
equivariant surjection:
\[
V_{(m-1)d} \otimes \HH^0(\OO_{\PP^1}(t)) \to \HH^0(\OO_{\PP^1}((m-1)d+t)),
\]
all the surjections for $t \ge 1$ being induced by the case $t=0$, which in turn is obvious.
Therefore for $t=0$ we have $\HH^0(\LL_{m,d}(-d))=0$ while, for $t=1$, using \eqref{tensor prod dec 2}, we get
\[
\HH^0(\LL_{m,d}(1-d)) \simeq V_{(m-1)d-1},
\]
which easily implies \eqref{LP1}.
Next, we rewrite the exact sequence defining $\V_{m,d}$ as:
\[
0 \to \V_{m,d} \to V_{md-2}\otimes \OO_{\PP^1} \to V_{(m-1)d-1}
\otimes \OO_{\PP^1}(d-1) \to 0.
\]

\subsubsection{The pulled-back kernel bundle}
Next we want to prove:
\begin{equation}
  \label{V}
\V_{m,d} \simeq V_{d-2} \otimes \OO_{\PP^1}((1-m)d).  
\end{equation}

To check this, note that \eqref{tensor prod dec 2} gives, via
the same argument that we mentioned to define $\Psi_{m,d}$, an
equivariant map:
\[
\vartheta_{m,d} : V_{d-2} \otimes \OO_{\PP^1}((1-m)d) \to V_{md-2} \otimes \OO_{\PP^1}.
\]
The highest weight vector defining the irreducible representation
$V_{d-2}$ as a direct summand of $V_{(m-1)d}\otimes V_{md-2}$ is the following:
\[
\sum_{i=0}^{(m-1)d} (-1)^i \binom{(m-1)d}{i} \mathsf{x_1}^{i} \mathsf{x_2}^{(m-1)d-i} \otimes \mathsf{x_1}^{md-2-i} \mathsf{x_2}^{i}.
\]
Acting $k$ times with $Y$ gives
\[
\sum_{i=0}^{(m-1)d} (-1)^i \binom{(m-1)d}{i} \mathsf{x_1}^{i} \mathsf{x_2}^{(m-1)d-i} \otimes \mathsf{x_1}^{md-2-i-k} \mathsf{x_2}^{i+k}.
\] 
Hence, the map $\vartheta_{m,d}$ is described by the following matrix:
\begin{center}
\begin{tikzpicture}[scale=0.5]
\draw [thick] (0.,0.3) -- (-0.3,0.3) -- (-0.3,-16.1) -- (0.0,-16.1);
\draw [thick] (8.8,0.3) -- (9.1,0.3) -- (9.1,-16.1) -- (8.8,-16.1);

\node at (6,-2) [] {\Large $0$};
\node at (2.8,-13.8) [] {\Large $0$};

\draw (0.0,0) -- (1,0) -- (1,-8) -- (0.0,-8) -- cycle;
\node at (0.5,-4) [] {$C$};
\draw (0.0+1.3,0-1.3) -- (1+1.3,0-1.3) -- (1+1.3,-8-1.3) -- (0.0+1.3,-8-1.3) -- cycle;
\node at (0.5+1.3,-4-1.3) [] {$C$};

\draw (0.0+2*1.3,0-2*1.3) -- (1+2*1.3,0-2*1.3) -- (1+2*1.3,-8-2*1.3) -- (0.0+2*1.3,-8-2*1.3) -- cycle;
\node at (0.5+2*1.3,-4-2*1.3) [] {$C$};

\node at (0.5+3.5*1.3,-4-3.5*1.3) [] {$\ddots$};

\draw (0.0+5*1.3,0-5*1.3) -- (1+5*1.3,0-5*1.3) -- (1+5*1.3,-8-5*1.3) -- (0.0+5*1.3,-8-5*1.3) -- cycle;
\node at (0.5+5*1.3,-4-5*1.3) [] {$C$};

\draw (0.0+6*1.3,0-6*1.3) -- (1+6*1.3,0-6*1.3) -- (1+6*1.3,-8-6*1.3) -- (0.0+6*1.3,-8-6*1.3) -- cycle;
\node at (0.5+6*1.3,-4-6*1.3) [] {$C$};

\node at (18,-7.5) [] {with\ $C = {\small \left[
\begin{array}{c}	
\mathsf{x_2}^{(m-1)d}  \\
-\binom{(m-1)d}{1} \mathsf{x_1} \mathsf{x_2}^{(m-1)d-1} \\
\\
\vdots \\
\\
(-1)^i \binom{(m-1)d}{i} \mathsf{x_1}^i \mathsf{x_2}^{(m-1)d-i} \\
\\
\vdots\\
\\
(-1)^{(m-1)d-1}\binom{(m-1)d}{(m-1)d-1} \mathsf{x_1}^{(m-1)d-1}  \mathsf{x_2}\\
(-1)^{(m-1)d} \mathsf{x_1}^{(m-1)d} \\
\end{array}
\right]}$};
\end{tikzpicture}
\end{center}

It is clear that this matrix is injective at one (and hence at any) point of $\PP^1$.

\bigskip

Now, again \eqref{tensor prod dec 2} says that $V_{(m-1)d-1}\otimes
V_{md-1} \simeq V_{2md-d-2}\oplus \cdots \oplus V_d$, so the
representation $V_{d-2}$ does not occur in this direct sum. Then,
by the same argument as in end of the proof of Lemma \ref{composizione nulla}, 
the composition of maps
\[V_{d-2} \otimes \OO_{\PP^1}((1-m)d) \to V_{md-2}\otimes \OO_{\PP^1} \to V_{(m-1)d-1} \otimes \OO_{\PP^1}(d-1)\]
is zero.
Therefore, we get an injective map $V_{d-2} \otimes \OO_{\PP^1}((1-m)d)
\to \V_{m,d}$.
Therefore, this map is an isomorphism, because 
it is clear from exactness of the equivariant complex that $V_{d-2} \otimes \OO_{\PP^1}((1-m)d)$ and $\V_{m,d}$ are vector
bundles on $\PP^1$ having the same rank and first Chern class.

\subsubsection{Proof of stability}

The isomorphism \eqref{V} implies plainly that $\W_{m,d}$ is slope-semistable.
Indeed, the pull-back of a subbundle of $\W_{m,d}$ having strictly
higher slope than $\W_{m,d}$ would be a subbundle of $V_{d-2} \otimes
\OO_{\PP^1}(d(1-m))$ again with strictly higher slope, which is absurd
because this bundle is semistable on $\PP^1$.

Now we prove that $\W_{m,d}$ is actually slope-stable. Indeed, assume
$\W_{m,d}$ has a non-trivial filtration by subsheaves, whose
quotients are slope-stable sheaves. We write the associated graded
object in the form:
\[
\gr(\W_{m,d}) = \bigoplus_{i=1}^s \FF_i^{\oplus r_i},
\]
where $r_i$ are positive integers and  $\FF_1,\ldots,\FF_s$ are stable
sheaves on $\PP^d$, with $\FF_i$  not isomorphic to $\FF_j$ for $i \ne j$.
For any $g \in G$, we get a linear automorphism of $\PP^d$ which we
still denote by $g$ and an
isomorphism $g^*(\W_{m,d}) \simeq \W_{m,d}$ which in turn induces an
automorphism $\gr(\W_{m,d}) \to \gr(\W_{m,d})$. Since $\FF_i \not\simeq \FF_j$ for $i \ne j$, we get isomorphisms
$\FF_i^{\oplus r_i} \to \FF_i^{\oplus r_i}$ and since each $\FF_i$ is
stable all such morphisms are of the form $g_i \otimes \id_{\FF_i}$,
for some linear isomorphism $g_i : \C^{r_i} \to \C^{r_i}$.
In other words, there are $G$-representations $R_1,\ldots,R_s$ such
that:
\[
\gr(\W_{m,d}) \simeq \bigoplus_{i=1}^s R_i \otimes \FF_i.
\]

Let us now look at the pull-back of the filtration to $\PP^1$ via
$v_{1,d}$. By semi-stability of $\V_{m,d}$ and homogeneity of the $\FF_i$ each
$\FF_i$ pulls back to $T_i \otimes \OO_{\PP^1}(d(1-m))$, for some
$G$-representation $T_i$.  So
the associated graded object is of the form 
\[
\gr(\V_{m,d}) \simeq \bigoplus_{i=1}^s R_i \otimes T_i \otimes
\OO_{\PP^1}(d(1-m)).
\]

Therefore, \eqref{V} gives $s=1$ and $R_1=V_{d-2}$ and $T_1\simeq \C$,
or $T_1\simeq V_{d-2}$ and $R_1 \simeq \C$. We want to exclude the
former case, the latter corresponding the fact that $\W_{m,d}$ is
stable.
But if $R_1=V_{d-2}$ and $T_1 \simeq \C$,
we get that $\FF_1$ is a line bundle on $\PP^d$ which implies immediately $\W_{m,d} \simeq
V_{d-2} \otimes \OO_{\PP^d}(1-m)$.
However, a straightforward computation on the equivariant complex shows
that the second Chern class of $\W_{m,d}(m-1)$ is non-zero, so this is
impossible (see the next subsection for more details on the Chern
classes of $\W_{m,d}$). The proof of Theorem \ref{risultato SL2} is thus complete.

\begin{rem}
For $d=3$ our result agrees with the resolution-theoretic approach to the construction of $\SL_2(\C)$-equivariant instantons achieved in \cite{dani_ist_omog}. In fact, in that paper the classification of instantons on $\PP^3$ which are invariant for any linear action of $\SL_2(\C)$ on $\PP^3$ was completed. No such classification is available to our knowledge for equivariant vector bundle on higher-dimensional projective spaces.
\end{rem}

We postpone to \S \ref{varie} a more detailed study of our $\SL_2(\C)$-equivariant bundles, where we will show in Proposition \ref{new} that, as soon as $d \ge 4$ and $m \ge 3$, the $\W_{m,d}$'s are not isomorphic to any previously known rank $d-1$ bundle on $\PP^d$.

%%%%%%%%%%%%%%%%%%%%%%%%%%%%%%%%%%%%%%%%%%%%%%%%%%%%%%%%%
\section{The general case} \label{G}
%%%%%%%%%%%%%%%%%%%%%%%%%%%%%%%%%%%%%%%%%%%%%%%%%%%%%%%%%

This section is devoted to the proof of Theorem \ref{risultato
  generale}; the steps are similar to what is done in the previous
case of $G=\SL_2(\C)$, however the situation is much more complicated this
time, as we have to deal with representations whose highest weight is not just a multiple of
$\lambda_1$, but a linear combination of $\lambda_1$, $\lambda_2$, and
in some cases also $\lambda_3$. We exploit the theory of representations of Lie groups and of quiver
representations. 

We illustrate
the proof in the next subsections; we start by recalling some facts
about the equivalence of categories between $G$-homogeneous bundles
and quiver representations. Then, similarly to what we did in the
previous section, we work on the pulled-back bundle. 

\subsection{Quiver representations} 
The category of $G$-homogeneous bundles on $\PP^n$ is naturally equivalent to the category of
representations of the parabolic subgroup $P$ of $G$ such that $\PP^n
\simeq G/P$; this equivalence sends indecomposable bundles to
irreducible representations. 

The semisimple part of $P$, denoted by $R$, is a copy of
$\SL_{n}(\C)$; an irreducible representation of $R$ can be extended to
$P$ by letting the unipotent part act trivially. Homogenous bundles
corresponding to these irreducible representation of $R$ are called
\textit{completely reducible}. A completely reducible bundle is
canonically determined by its $P$-dominant weight, which is of the
form  
$\lambda=a_1\lambda_1+a_2\lambda_2\cdots+a_n\lambda_n$, where $a_1 \in
\Z$ and $a_2,\ldots,a_n \in \N$, and is therefore denoted by
$\EE_\lambda \otimes \C^{k(\lambda)}$, where the tensor product
accounts for its multiplicity.  If the multiplicity is $1$ we simply
write $\EE_\lambda$ for $\EE_\lambda \otimes \C$.  

The natural numbers $a_i$ for $i \ge 2$ are the coefficients in the
basis $(\lambda_2,\ldots,\lambda_n)$ of the dominant weight of the
representation of $R$ corresponding to the bundle. Notice that, given
a weight $\lambda$ as above, if we set $\mu_i := a_i+\ldots+a_n$ for
$i=1,\ldots,n$, we have: 
\[
\mbox{$\EE_\lambda \simeq \Gamma^\mu \Omega_{\PP^n} \otimes \OO_{\PP^n}{\big(}\sum_{i=1}^n \mu_i {\big)}$,}
\]
where $\Gamma^{\mu}$ is the Schur functor defined by the partition
$\mu=(\mu_2 \ge \ldots \ge \mu_n)$, and $\Omega_{\PP^n}$ the cotangent
bundle. In particular,  
%writing $\lambda=a\lambda_1 + \mu$, with $a \in \Z$, and $\mu
%=a_2\lambda_2+\cdots+a_n\lambda_n$, $a_i \in \N$, one has 
\[
\EE_{k\lambda_1+\nu} \simeq \EE_\nu(k), \qquad \EE_{\lambda_n} =
T_{\PP^n}(-1), \qquad \hbox{and} \quad \EE_{\lambda_2} = \Omega_{\PP^n}(2).
\]

A fundamental fact is that, when $\lambda$ is also $G$-dominant, that is, when $a_1 \ge 0$ in the definition of $\lambda$ above, the bundle $\EE_\lambda$ is globally generated and satisfies
\[
\HH^0(\PP^n,\EE_{\lambda})^G \simeq V_\lambda.
\]

Any $G$-homogeneous bundle $\EE$ on $\PP^n$ admits a filtration of the form:
\begin{equation}\label{filtration}
0 = \EE^s \subset \EE^{s-1} \subset \cdots \subset \EE^{1} \subset \EE^0 = \EE,  
\end{equation}
where $\FF^k=\EE^{k-1}/\EE^{k}$ is completely reducible for all $k=1,\ldots,s$.
We write:
\[
\gr(\EE)=\bigoplus_{k=1}^s \FF^k,
\]
the \emph{graded homogeneous vector bundle} associated with $\EE$. 

We can compute the graded bundle associated with $\Gamma^\sigma V \otimes \OO_{\PP^n}$ for any partition $\sigma=(\sigma_1 \ge \ldots \ge \sigma_n)$. The result is probably folklore, but the reader can find a detailed proof in \cite[Proposition 3]{re_pp_bundles}. We have that:
\begin{equation}\label{graduato schur}
\gr(\Gamma^\sigma V \otimes \OO_{\PP^n})= \bigoplus_{\nu} \Gamma^\nu \Omega_{\PP^n} \otimes \OO_{\PP^n}(|\sigma|),
\end{equation}
the sum extended over all $\nu$ obtained from $\sigma$ by removing any number of boxes from its Young diagram, with no two in any column. 

\medskip

There is a third category that is equivalent to the two above, namely
the category of finite dimensional representations of a certain quiver
$\Q_{\PP^n}$ with relations, see \cite{BK, OR, borasections}, and the
already quoted \cite{re_pp_bundles} among many other references. The
vertices of $\Q_{\PP^n}$ are given by $R$, the semisimple part of $P$,
and correspond to completely reducible bundles $\EE_\lambda$; we label
them with their highest weight $\lambda$. The arrows encode the
unipotent action, that is, the data of the extensions of $\FF^k$ and
$\FF^j$ induced by the filtration \eqref{filtration}, for
$k,j=1,\ldots,s$; we label them with the weights $\xi_1,\ldots,\xi_n$
of the cotangent bundle $\Omega_{\PP^n}$. The relations on
$\Q_{\PP^n}$ are the commutative ones.

\subsection{Exactness of the equivariant complex} As noticed before, Lemmas \ref{composizione nulla} and \ref{psi surj} together entail that the sequences \eqref{ex seq} and  \eqref{restr ex seq} are complexes, exact at the sides; once again, we need to show that exactness holds in the middle, as well. We work on the restricted
sequence \eqref{restr ex seq}, then extend the result to \eqref{ex seq} by semicontinuity. 

\smallskip

We split \eqref{restr ex seq} into two short exact sequences:
\begin{equation}\label{A}
0 \to \V_{m,d} \to  V_{(md-2)\lambda_1 + \lambda_2} \ot \OO_{\PP^n} \xrightarrow{\phi_{m,d}} \LL_{m,d},
\end{equation}
where $\LL_{m,d}$ is an $G$-homogeneous bundle on $\PP^n$,
defined as the kernel of $\psi_{m,d}$ by:
\begin{equation}
  \label{K}
0 \to \LL_{m,d} \to V_{(m-1)d} \ot \OO_{\PP^n}(d) \xrightarrow{\psi_{m,d}} \OO_{\PP^n}(md) \to 0.  
\end{equation}
 
We study the bundle $\LL_{m,d}$ more in detail; rewrite \eqref{K} twisted by $\OO_{\PP^n}(-d)$:
\begin{equation} \label{Kd}
0 \to \LL_{m,d}(-d) \to V_{(m-1)d} \ot \OO_{\PP^n} \to \OO_{\PP^n}((m-1)d) \to 0.  
\end{equation}

Two immediate remarks are in order. First, the bundle  $\LL_{m,d}^*(d)$ is
generated by an irreducible module of global sections, and this implies that $\LL_{m,d}$ is indecomposable. Second, the map induced on global sections by the surjection 
$V_{(m-1)d} \ot \OO_{\PP^n} \twoheadrightarrow \OO_{\PP^n}((m-1)d)$ is an isomorphism, and from this we deduce the cohomology vanishings:
\[
\HH^i(\PP^n,\LL_{m,d}(-d))=0, \qquad \mbox{for all $i \in \N$.}
\]

%From the Euler sequence we can compute the graded bundle
%\[
%\gr(V \ot \OO_{\PP^n}) = \EE_{\lambda_2-\lambda_1} \oplus \EE_{\lambda_1}.
%\]
%Also, for $\ell \ge 0$, we have $V_\ell \simeq S^\ell V$ and $S^\ell
%\EE_{\lambda_1} \simeq \EE_{\ell \lambda_1}$. Therefore, 
From formula \eqref{graduato schur} we compute:
\begin{equation} \label{Vell}
\gr(V_\ell \ot \OO_{\PP^n}) \simeq \bigoplus_{k=0}^\ell \EE_{(\ell-2k)\lambda_1+k\lambda_2}.  
\end{equation}

Combining \eqref{Kd} and \eqref{Vell}, and recalling that $\OO_{\PP^n}(m)=\EE_{m\lambda_1}$, we compute the graded bundle associated to $\LL_{m,d}(-d)$, and from it:
\begin{equation}\label{grad L}
\gr(\LL_{m,d}) \simeq \bigoplus_{k=1}^{(m-1)d} \EE_{(md-2k)\lambda_1 +k\lambda_2}.
\end{equation}

Therefore, in the filtration \eqref{filtration} of $\LL_{m,d}$, the index $s$ is $s=(m-1)d$, and the completely reducible quotients $\FF^k$ are:
\[
\FF^k=\LL_{m,d}^{k-1}/\LL_{m,d}^k \simeq \EE_{((m-1)d-2k)\lambda_1+k\lambda_2}.
\]

Notice that all summands in \eqref{grad L} are completely reducible homogeneous bundles and they all appear with multiplicity $1$ in the decomposition. Moreover the first summand (i.e.~$k=1$) is the only one satisfying:
\[
\HH^0(\PP^n,\EE_{(md-2)\lambda_1+\lambda_2}) \simeq V_{(md-2)\lambda_1+\lambda_2}.
\]

Let us now look at the associated quiver representation; we will denote by $[\EE]$ the representation associated to a homogeneous bundle $\EE$. Computing the action of the nilpotent part of the parabolic $P$, we see that the support of both quiver representations $[V_{(m-1)d} \ot \OO_{\PP^n}]$ and $[\LL_{m,d}]$ is connected with all arrows in the same direction, namely the one associated with the first weight $\xi_1$ of the cotangent bundle $\Omega_{\PP^n}$. In other words, the support of these two quiver representations is an \emph{$A_p$-type quiver} contained in:

\[
\xymatrix@R-4ex@C-2ex{
\:_{\scriptsize \substack{md\lambda_1\\}}& 
\:_{\scriptsize \substack{(md-2)\lambda_1\\+\lambda_2}}& 
\:_{\scriptsize \substack{(md-4)\lambda_1\\+2\lambda_2}}&
\ldots&
\:_{\scriptsize \substack{(2d+2-md)\lambda_1\\+(md-d-1)\lambda_2}}& 
\:_{\scriptsize \substack{(2d-md)\lambda_1\\+(md-d)\lambda_2}}\\
\bullet \ar[r]_-{\xi_1}&\bullet\ar[r]&\bullet\ar[r]_-{\xi_1}&\cdots\cdots
\ar[r]_-{\xi_1}&\bullet\ar[r]&\bullet
}
\]
The three quiver representations $[V_{(m-1)d} \ot \OO_{\PP^n}]$, $[\LL_{m,d}]$, and $[\OO_{\PP^n}(md)]$ associated to the bundles in \eqref{K} have dimension vector $[1\:1\:1\ldots1\:1]$, $[0\:1\:1\ldots1\:1]$, and 
$[1\:0\:0\ldots 0\:0]$ respectively.

\smallskip

Now consider the two equivariant morphisms 
\[
\xymatrix@R-4ex@C-2ex{ V_{(md-2)\lambda_1+\lambda_2} \ot \OO_{\PP^n} \ar[rr]^-{\phi_{m,d}} && \LL_{m,d} \ar@{>>}[rr]^-{\pi_{m,d}}& &\EE_{(md-2)\lambda_1+\lambda_2}
%\\[1\:1\:1\ldots 1\:1]  \ar[rr]&&[0\:1\:1\ldots1\:1]\ar@{>>}[rr]&&[0\:1\:0\ldots 0\:0]
}
\]
and their composition. The induced map on global sections is an equivariant map of irreducible representations $V_{(md-2)\lambda_1+\lambda_2} \to V_{(md-2)\lambda_1+\lambda_2}$, so it is either zero or an isomorphism. If it were zero, then $\phi_{m,d}$
%$V_{\lambda_2+(md-2)\lambda_1} \ot \OO_{\PP^n} \to \LL_{m,d}$
would factor through the kernel $\KK_{m,d}$ of $\pi_{m,d}$, whose graded object is  
%\[
%\xymatrix@R-3ex@C-2ex{ V_{(md-2)\lambda_1+\lambda_2} \ot \OO_{\PP^n} \ar[dr]\ar[rr]^-{\phi_{m,d}} && \LL_{m,d} \ar@{>>}[rr]& &\EE_{(md-2)\lambda_1+\lambda_2}\\
%&\KK_{m,d} \ar[ur]&&&}
%\]
%The graded object of the kernel is:
\begin{equation}\label{grK}
\gr(\KK_{m,d}) \simeq \bigoplus_{k=2}^{(m-1)d} \EE_{(md-2k)\lambda_1+k\lambda_2}.  
\end{equation}
Hence we would get a non-zero map from an irreducible $G$-module to the space of global sections of a bundle whose graded object has no summands with this particular $G$-module as its space of global sections, a contradiction. 

We conclude that the composition $\pi_{m,d} \circ \phi_{m,d}$ induces an isomorphism on global sections. Therefore, it is a surjective morphism of sheaves because  $\EE_{(md-2)\lambda_1+\lambda_2}$ is completely reducible of multiplicity $1$ and associated with a dominant weight, and is thus globally generated.

\smallskip

Now let us finally prove that $\phi_{m,d}: V_{(md-2)\lambda_1+\lambda_2} \ot \OO_{\PP^n}
\to \LL_{m,d}$ is surjective, and hence that sequence \eqref{restr ex seq} is exact. Setting  $Q_{m,d}=\coker \phi_{m,d}$, we have the following commutative exact diagram:
\[
\xymatrix@-1ex{
0 \ar[r]  &\KK_{m,d}' \ar[d] \ar[r]& V_{(md-2)\lambda_1+\lambda_2} \ot \OO_{\PP^n} \ar[d]^{\phi_{m,d}} \ar[r]& \EE_{(md-2)\lambda_1+\lambda_2} \ar[r] \ar@{=}[d]& 0\\
0 \ar[r] & \KK_{m,d} \ar[d] \ar[r] & \LL_{m,d} \ar[r]^-{\pi_{m,d}} \ar[d] & \EE_{(md-2)\lambda_1+\lambda_2} \ar[r] & 0\\
 &Q_{m,d}'\ar[r] & Q_{m,d}}
\]
where all maps are equivariant, $\KK_{m,d}'$ is defined as kernel of the surjection $\pi_{m,d} \circ \phi_{m,d}$ above, and $Q_{m,d}'$ is the cokernel of the induced morphism $\KK_{m,d}' \to \KK_{m,d}$. The snake lemma implies that the map $Q_{m,d}' \to Q_{m,d}$ is an isomorphism.

Since $Q_{m,d}'$ is an equivariant quotient of $\KK_{m,d}$ and in view of
\eqref{grK}, the completely reducible
bundles occurring in $\gr(Q_{m,d}')$ are all of the form
$\EE_{(md-2k)\lambda_1+k\lambda_2}$, for some $k \in \{2,\ldots,(m-1)d\}$.
In particular for some $k \ge 2$ we get a surjection:
\begin{equation}\label{surjection}
\xymatrix@R-4ex@C-2ex{\LL_{m,d} \ar@{>>}[r] &\EE_{(md-2k)\lambda_1+k\lambda_2}.\\
[0,1,1,1,\ldots,1,1]&[0,0,0,1,\ldots,0,0]}
\end{equation}

Set $\GGG^k:=\LL_{m,d}/\LL^k$, with graded bundle:
\[
\gr(\GGG^k) \simeq \bigoplus_{j = 1}^k \FF^j.
\]
We observe that the surjection \eqref{surjection} necessarily factors through:
\[
\LL_{m,d} \to \GGG^k \to \FF^k \simeq  \EE_{(md-2k)\lambda_1+k\lambda_2},
\]
indeed the morphism $\LL_{m,d} \to \FF^k$ restricts to zero to $\LL_{m,d}^k$ because
\[
\gr(\LL_{m,d}^k) \simeq \bigoplus_{\ell = k+1}^{(m-1)d} \FF^\ell
\]
and none of these summands $\FF^\ell$ has non-trivial maps to $\FF^k$ for $k \ne \ell$.

Now, by definition of the filtration, we have an exact sequence
\[
0 \to \FF^k \to \GGG^k \to \GGG^{k-1} \to 0,
\]
and composing the injection $\FF^k \hookrightarrow \GGG^k$ with the surjection
$\GGG^k \twoheadrightarrow \FF^k$ we obtain an isomorphism. Indeed, if the composition were zero then the same would be true for map $\GGG^k \to \FF^k$, as no other summand of $\gr(\GGG^k)$ maps non-trivially to $\FF^k$.

We conclude that $\GGG^k \simeq \FF^k \oplus \GGG^{k-1}$, 
%Now we observe that:
%\[ \Ext^1_{\PP^n}(\GGG^{k-1},\LL_{m,d}^k)^{G}=0 \]
%because, if $j\le k-1$ and $\ell \in \{k+1,\ldots,(m-1)d\}$, we have: 
%\[ \Ext^1_{\PP^n}(\FF^\ell,\FF^j)^{G}=0. \]
and this in turn entails a splitting $\LL_{m,d} \simeq \GGG^{k-1} \oplus \EE^{k-1}$, for some $k \ge 2$, which is a contradiction because we have already proved that $\LL_{m,d}$ is indecomposable.

\subsection{Indecomposability and rank computation} Notice that the vector bundle $\V_{m,d}$ is indecomposable, again because its dual $\V_{m,d}^\ast$ is generated by an irreducible module of global sections, and thus $\W_{m,d}$ is indecomposable, as well.

Moreover, $\rk(\W_{m,d})=\rk(\V_{m,d})$, and this rank, using the dimension formula \cite[p. 224]{FH}, equals:
%\[\dim V_{\ell_1\lambda_1+ \ldots + \ell_{n-1}\omega_{n-1}}= \prod_{1 \le i < j \le n} \frac{\ell_i+\ldots+\ell_{j-1}+j-i {j-i}\]
\[
\dim V_{(md-2)\lambda_1+\lambda_2} - \dim V_{(m-1)d} + 1= (m^2d^2-1) - \tbinom{(m-1)d+2}{2} + 1.
%= \frac{m^2d^2+2md^2-3md-d^2+3d-2}{2}.
\]

\subsection{Slope-stability and quiver $\mu$-stability} To conclude
the proof of Theorem \ref{risultato generale} we need to show that
$\W_{m,d}$ is slope-stable. Once again, we prove our result on the
closed orbit, and exploit the connection between slope-stability of a
homogeneous bundle and $\mu$-stability ``\`a la King'' of the
associated quiver representation, see \cite{king}.  

Given a homogenous vector bundle $\EE$ and its associated quiver representation $[\EE]$, we define
\[
\mu_{[\EE]}(-)=c_1(\gr (\EE)) \rk (-) - \rk(\gr (\EE))c_1(-),
\]
and call the representation $[\EE]$ \emph{$\mu$-stable} if for all
subrepresentations $[\EE']$ one has that $\mu_{[\EE]}(\gr(\EE')) \ge
0$ and $\mu_{[\EE]}(\gr(\EE')) = 0$ if and only $[\EE']$ is either
$[\EE]$ or $[0]$.  

From \cite[Theorem 7.2]{OR} we learn that $[\V_{m,d}]$ is $\mu$-stable
if and only if $\V_{m,d}=W \otimes \FF$, with $\FF$ a slope-stable
homogeneous bundle, and $W$ an irreducible $G$-module. Now, if we had
$\V_{m,d}=W \otimes \FF$, then the resolution of $\V_{m,d}$ (better
yet, of its dual) would be given by the resolution of $\FF$ tensored
by the irreducible representation $W$, a contradiction with
\eqref{restr ex seq}. Therefore to prove that $\V_{m,d}$ is
slope-stable we only need to show that the associated representation
$[\V_{m,d}]$ is $\mu$-stable. For this, we study $[\V_{m,d}]$ and its
subrepresentations.  

\medskip

Let us first suppose that $n \ge 3$. From the short exact sequence \eqref{A} we deduce the equality
\[
\gr(V_{(md-2) \lambda_1+\lambda_2} \otimes \OO_{\PP^n}) = \gr(\V_{m,d}) \oplus \gr(\LL_{m,d}).
\]

The graded bundle $\gr(\LL_{m,d})$ was already computed in \eqref{grad L}. Moreover, from formula \eqref{graduato schur} we have:
\begin{equation}\label{X}
\gr(V_{(md-2)\lambda_1+\lambda_2} \otimes \OO_{\PP^n}) \simeq
\bigoplus_{k=1}^{md-1} \left(\EE_{(md-1-2k)\lambda_1+(k-1)\lambda_2+\lambda_3} \oplus
\EE_{(md-2k)\lambda_1+k\lambda_2} \right),
\end{equation}
so all in all:
\begin{equation}\label{Y}
\gr(\V_{m,d})=\bigoplus_{k=1}^{md-1} \EE_{(md-1-2k)\lambda_1+(k-1)\lambda_2+\lambda_3} 
\oplus \bigoplus_{k=md-d+1}^{md-1} \EE_{(md-2k)\lambda_1+k\lambda_2},
\end{equation}
where we remark that all summands in both \eqref{X} and \eqref{Y} appear with multiplicity $1$. 

The subquiver of $\Q_{\PP^n}$ corresponding to the support of the representations $[\V_{m,d}]$, 
$[V_{(md-2) \lambda_1+\lambda_2} \otimes \OO_{\PP^n}]$, and $[\LL_{m,d}]$ is contained in:

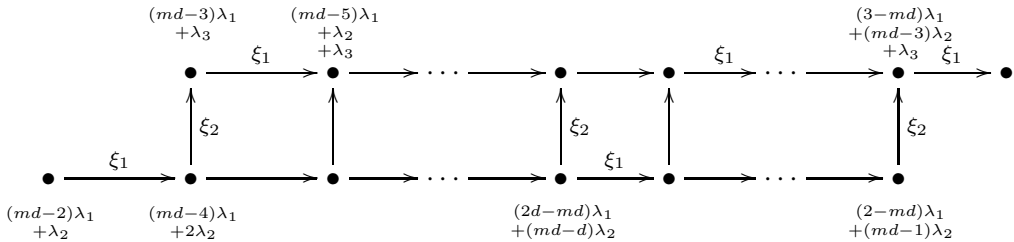
\[
\xymatrix@R-4ex@C-3ex{
&\:_{\scriptsize \substack{(md-3)\lambda_1\\+\lambda_3}}&
\:_{\scriptsize \substack{(md-5)\lambda_1\\+\lambda_2\\+\lambda_3}}&&&&&&
\:_{\scriptsize \substack{(3-md)\lambda_1\\+(md-3)\lambda_2\\+\lambda_3}}
\\
&
\bullet\ar[r]^{\xi_1}&
\bullet\ar[r]&
\cdots\ar[r]&
\bullet\ar[r]&
\bullet\ar[rr]^{\xi_1}&&
\cdots\ar[r]&
\bullet\ar[r]^{\xi_1}&
\bullet\\
&&&&&&&&&\\&&&&&&&&&\\
\bullet\ar[r]^-{\xi_1}&
\bullet\ar[r] \ar[uuu]_-{\xi_2}&
\bullet\ar[r] \ar[uuu]&
\cdots\ar[r]&
\bullet\ar[r]^-{\xi_1} \ar[uuu]_-{\xi_2}&
\bullet\ar[rr]\ar[uuu]&&
\cdots\ar[r]&
\bullet \ar[uuu]_-{\xi_2}&\\
\:^{\scriptsize \substack{(md-2)\lambda_1\\+\lambda_2}}&
\:^{\scriptsize \substack{(md-4)\lambda_1\\+2\lambda_2}}&&&
\:^{\scriptsize \substack{(2d-md)\lambda_1\\+(md-d)\lambda_2}}&&&&
\:^{\scriptsize \substack{(2-md)\lambda_1\\+(md-1)\lambda_2}}&
}\]

The short exact sequence \eqref{A} is associated to the following sequence of quiver representations:
\[
\xymatrix@R-4ex{\V_{m,d} \ar@{^{(}->}[r]&V_{(md-2) \lambda_1+\lambda_2} \otimes \OO_{\PP^n}
\ar@{>>}[r] &\LL_{m,d}\\
{\scriptsize{\left[\setlength\arraycolsep{2pt}\begin{array}{ccccccccc}
&1&1&\cdots&1&1&\cdots&1&1\\
0&0&0&\cdots&0&1&\cdots&1&
\end{array}\right]}}&
{\scriptsize{\left[\setlength\arraycolsep{2pt}\begin{array}{ccccccccc}
&1&1&\cdots&1&1&\cdots&1&1\\
1&1&1&\cdots&1&1&\cdots&1&
\end{array}\right]}}&
{\scriptsize{\left[\setlength\arraycolsep{2pt}\begin{array}{ccccccccc}
&0&0&\cdots&0&0&\cdots&0&0\\
1&1&1&\cdots&1&0&\cdots&0&
\end{array}\right]}}
}
\]

Computing the action of the nilpotent part of the parabolic, it is easy to see that all maps of type $\C \to \C$ in the quiver representations mentioned above are non-zero. Indeed, if any of them were zero, then the support of the quiver representation would disconnect and hence the bundle would decompose, which we already know is impossible; this is apparent for the $\xi_1$-type arrows, whereas for the $\xi_2$-type arrows one needs to remember the commutativity relations holding on $\Q_{\PP^n}$.

\smallskip

What do the subrepresentations of $[\V_{m,d}]$ look like? 
The graded bundle $\gr(\V_{m,d})$ has two types of summands; we call $\A_k := 
\EE_{(md-1-2k)\lambda_1+(k-1)\lambda_2+\lambda_3}$ the first type, with $k=1,\ldots,md-1$, and $\B_k:=\EE_{(md-2k)\lambda_1+k\lambda_2}$ the second one, $k=md-d+1,\ldots,md-1$. 
For the reader's convenience, we re-draw the support of $[\V_{m,d}]$:

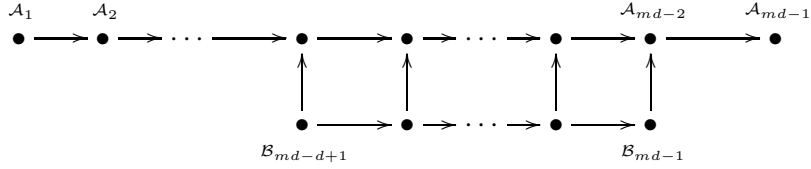
\[
\xymatrix@R-5ex@C-2.5ex{
&\:_{\scriptsize \substack{\A_1}}&
\:_{\scriptsize \substack{\A_2}}&&&&&&
\:_{\scriptsize \substack{\A_{md-2}}}&
\:_{\scriptsize \substack{\A_{md-1}}}
\\
&
\bullet\ar[r]&
\bullet\ar[r]&
\cdots\ar[r]&
\bullet\ar[r]&
\bullet\ar[r]&\cdots\ar[r]&
\bullet\ar[r]&
\bullet\ar[r]&
\bullet\\
&&&&&&&&&\\&&&&&&&&&\\ &&&&&&&&&\\
&&&&
\bullet\ar[r] \ar[uuuu]&
\bullet\ar[r]\ar[uuuu]&\cdots\ar[r]&
\bullet\ar[r]\ar[uuuu]&
\bullet \ar[uuuu]&\\
&&&&
\:^{\scriptsize \substack{\B_{md-d+1}}}&&&&
\:^{\scriptsize \substack{\B_{md-1}}}&
}\]

Any subrepresentation of $[\V_{m,d}]$ is either of the $A_p$-type: 
\[
{\scriptsize{\left[\setlength\arraycolsep{2pt}\begin{array}{ccccccccc}
0&0&\cdots&0&0&\cdots&0&0&1\\
&&&0&0&\cdots&0&0&
\end{array}\right]}}, \quad
{\scriptsize{\left[\setlength\arraycolsep{2pt}\begin{array}{ccccccccc}
0&0&\cdots&0&0&\cdots&0&1&1\\
&&&0&0&\cdots&0&0&
\end{array}\right]}}, \: \hbox{and so on until:} \:
{\scriptsize{\left[\setlength\arraycolsep{2pt}\begin{array}{ccccccccc}
1&1&\cdots&1&1&\cdots&1&1&1\\
&&&0&0&\cdots&0&0&
\end{array}\right]}},
\]
or else it is of the hook form:
\[
{\scriptsize{\left[\setlength\arraycolsep{2pt}\begin{array}{ccccccccc}
0&0&\cdots&0&0&\cdots&0&1&1\\
&&&0&0&\cdots&0&1&
\end{array}\right]}}, \quad
{\scriptsize{\left[\setlength\arraycolsep{2pt}\begin{array}{ccccccccc}
0&0&\cdots&0&0&\cdots&1&1&1\\
&&&0&0&\cdots&1&1&
\end{array}\right]}},\]
and so on until:
\[
[\V_{m,d}]={\scriptsize{\left[\setlength\arraycolsep{2pt}\begin{array}{ccccccccc}
1&1&\cdots&1&1&\cdots&1&1&1\\
&&&1&1&\cdots&1&1&
\end{array}\right]}},
\]
where again all maps of type $\C \to \C$ above are non-zero.

We have that $\A_k=\Gamma^{k,1}\Omega_{\PP^n} \otimes \OO_{\PP^n}(md)$, therefore, denoting by $\mu(\EE)$ the usual slope of sheaves, we compute that: 
\[
1\le i < j \le md -1 \Rightarrow \mbox{$\mu(\A_i)=md-\frac{(i+1)(n+1)}{n}$} > \mu(\A_j). 
\]

Therefore any nonzero subrepresentation  $[\EE']$ of $[\V_{m,d}]$ of
$A_p$-type will satisfy:
\[c_1(\gr (\V_{m,d})) \rk (\EE') - \rk(\gr (\V_{m,d})) c_1(\EE') > 0.\] 

For the hook type subrepresentations we compute that, since $\B_k=S^k\Omega_{\PP^n} \otimes \OO_{\PP^n}(md)$, \mbox{$\mu(\B_i)=md-\frac{i(n+1)}{n}$}. In particular:
\[
md-d+1\le i < j \le md -1 \Rightarrow \mu(\B_i)=\mu(\A_{i-1})> \mu_(\A_{j-1})=\mu(\B_j). 
\]
Therefore all these nonzero subrepresentation $[\EE']$ of $[\V_{m,d}]$ satisfy $c_1(\gr (\V_{m,d})) \rk (\EE') - \rk(\gr (\V_{m,d})) c_1(\EE') \ge 0$. The fact that the only hook subrepresentation
of $[\V_{m,d}]$ with $\gr(\EE') = \gr(\V_{m,d})$ is $[\V_{m,d}]$ itself concludes the proof of stability for $n \ge 3$.

\medskip

In the case $n=2$,  all summands of type $\A_k$ in formulas \eqref{X} and \eqref{Y} must be substituted with summands of type 
$\EE_{(md-1-2k)\lambda_1+(k-1)\lambda_2}$, and the quiver representations (and subrepresentations) look the same as in the previous case. Also, $\mu(\B_k)=md-\frac{3}{2}k$ and 
\mbox{$\mu(\A_k)=md -\frac{3}{2}(k-1)$}, therefore the same argument as above applies.

\smallskip

This concludes the proof of Theorem \ref{risultato generale}.

%%%%%%%%%%%%%%%%%%%%%%%%%%%%%%%%%%%%%%%%%%%%%%%%%%%%%%%%%
\section{Further remarks} \label{varie}
%%%%%%%%%%%%%%%%%%%%%%%%%%%%%%%%%%%%%%%%%%%%%%%%%%%%%%%%%

Let us point out some cohomological features of our equivariant bundles. 
We denote by $\HH^p_*(\FF)$ the cohomology module $\oplus_{t\in
  \Z}\HH^p(\PP^N,\FF(t))$ of a coherent sheaf $\FF$ over $\PP^N$.
This is an artinian module over the polynomial ring $R=\C[x_0,\ldots,x_N]$ if $\FF$ is locally free and $0 < p < N$.
Recall that, once the values of $n$ and $d$ are fixed, we have $N=\binom{d+n}{n}-1$ and
for all $m \ge 1$ the vector bundles $\W_{m,d}$ are defined on
$\PP^N$. Their cohomology modules are $G$-homogeneous.
We assume throughout that $N \ge 3$ and $m \ge 2$.

\begin{lem} \label{ciclico}
  The  $R$-module $\HH^2_*(\W_{m,d})$ is cyclic, generated in degree $-m$, with $V_{(m-1)d}$ as space of
  minimal degree relations in degree $1$. Also, $\HH^p_*(\W_{m,d})=0$ for $p=3,\ldots,N-1$.
\end{lem}

\begin{proof}
  We split the the equivariant complex \eqref{ex seq} into two short exact sequences and put 
  $\cZ_{m,d}=\im(\Phi_{m,d})$. It is easy to see that $\HH^p_*(\cZ_{m,d})=0$ for $p=2,\ldots,N-1$. Moreover, we deduce that $\HH^1(\cZ_{m,d}(t))=0$ for $t \ge -m-1$, and that $R(m)\simeq \HH^0_*(\OO_{\PP V_d}(m))$  surjects onto
  $\HH^1_*(\cZ_{m,d})$, so that this module is cyclic, generated in degree $-m$.

  Next, we have $\HH^0(\cZ_{m,d}(t))=0$ for $t\le -2$ and thus, for $t=-1$,
  we get that $\HH^0(\cZ_{m,d}(-1))$ is the kernel of the map $V_{(m-1)d} \to S^{m-1} V_d$ induced by $\Psi_{m,d}$. 
  By construction this map is injective, so $\HH^0(\cZ_{m,d}(-1))=0$, and therefore $\HH^1(\cZ_{m,d}(-1))$ is the quotient of $S^{m-1} V_d$ by $V_{(m-1)d}$, which is to say, the module of relations of $\HH^1_*(\cZ_{m,d})$ contains 
  $V_{(m-1)d}$, sitting in degree $1$, and no relations of smaller degree.

  Finally, it is clear that $\HH^p_*(\W_{m,d}) \simeq \HH^{p-1}_*(\cZ_{m,d})$ for all $p=2,\ldots,N-1$, so the lemma is proved.
\end{proof}

\medskip

We now focus our attention on the case of binary forms, when $n=2$ and $N=d$; as we mentioned in the introduction, few classes of indecomposable vector bundles of rank $d-1$ on $\PP^d$ are known, namely the classical mathematical instantons, Tango bundles, their generalizations by Cascini \cite{cascini:tango} and Bahtiti \cite{bahtiti-correlation,bahtiti:tango,bahtiti:2n+1}, and the Sasakura bundle of rank $3$ on $\PP^4$ \cite{anghel:sasakura}. 

In what follows we show that as soon as $d \ge 4$ our equivariant
bundles are new, except for a single case that we illustrate. Set
$N_{m,d,k}=\binom{(m-1)d+k}{k}$.

\smallskip

\begin{lem}\label{classi}
The normalized bundle associated to $\W_{m,d}$ is $\NN_{m,d}=\W_{m,d}(m-1)$ with Chern classes:
%\[\mbox{$c_k(\W_{m,d})= \frac{(1-m)(d-k) }{k!} \prod_{i=1}^{k-1} (d-i-md)$}.\]
%%% \[
%%% \mbox{$c_k(\mathcal{N}_{m,d}) =  (-1)^k \binom{(m-1)d+k}{k} m^k + \sum_{i=0}^{k-1} (-1)^i  \frac{(m-1)d + (k-i)}{md-(k-i)} \binom{md-1}{k-i} \binom{(m-1)d+i}{i} (m-1)^{k-i-1} m^{i+1}$.}
%%% \]
\[
\mbox{$c_k(\mathcal{N}_{m,d}) =  (-1)^k N_{m,d,k} m^k + \sum_{i=0}^{k-1} (-1)^i  \frac{(m-1)d + (k-i)}{md-(k-i)} \binom{md-1}{k-i} N_{m,d,i} (m-1)^{k-i-1} m^{i+1}$.}
\]
\end{lem}

\begin{proof}
The statement follows from splitting \eqref{ex seq sl2} into two short exact sequences followed by a fairly cumbersome computation.
\end{proof}

\begin{prop}\label{new}
 For $n=2$, and $d \ge 4$, the  bundle $\W_{m,d}$ is not isomorphic to any of the following bundles, even up to dualizing and taking pull-backs by finite self-maps of $\PP^d$:
 \begin{enumerate}[label=\roman*)]
 \item \label{insta} mathematical instanton bundles;
 \item \label{tango} weighted generalized Tango bundle, except if $m=2$, in which case $\W_{2,d}^*(-1)$ is a Tango bundle;
 \item \label{sasa} the Sasakura bundle, for $d=4$.
\end{enumerate}
 \end{prop}

\begin{proof}
Let us prove \ref{insta}.
The odd Chern classes of mathematical instanton bundles vanish; on the other hand, from Lemma \ref{classi} above we get that:
\[
\mbox{$c_3(\NN_{m,d})=-\frac{(d-3)}{3}m(m-1)(2m-1)$,}
\] 
which does not vanish for $d \ge 4$ and $m \ge 2$. 
%\begin{prop}
% For $n=2$ and $d \ge 4$, the vector bundle $\W_{m,d}$ is not isomorphic to a weighted generalized Tango bundle, even up to dualizing and taking pull-backs by finite self-maps of $\PP^d$, except if $m=2$, in which case 
%  $\W_{2,d}^*(-1)$ is a Tango bundle. 
%\end{prop}

\bigskip

%\begin{proof}
 Let us now prove \ref{tango}. Given integers $\gamma > 0$ and $\alpha \ge \beta$, we refer to
  $\FF_{\gamma,\alpha,\beta}$ as the vector bundle arising from Bahtiti's construction of generalized weighted 
  Tango bundles; this includes Cascini's weighted Tango bundles, which are indeed homogeneous for
  $\SL_2(\C)$ acting on the space of binary forms.
  Let us recall that $\FF=\FF_{\gamma,\alpha,\beta}$ fits into the long
  exact sequence:
  \[
  0 \to \OO_{\PP^d}(-3\gamma) \to
  \bigoplus_{i=0}^d\OO_{\PP^d}(d\alpha+i(\beta-\alpha)-2\gamma) \to
  \bigoplus_{i=0}^{2d-1}\OO_{\PP^d}(2d\alpha+i(\beta-\alpha)-\gamma) \to \FF \to 0.
  \]
  The image of the middle map is usually denoted by $Q_{\gamma,\alpha,\beta}(-2\gamma)$.

  We borrow the notation from the proof of Lemma \ref{ciclico} and use essentially the same argument and the fact that
  $d\ge 4$, to show that $\HH^2_*(Q_{\gamma,\alpha,\beta})=0$ and 
  consequently $\HH^1_*(\FF)=0$. Also, by stability of $\W_{m,d}$ we
  have $\HH^0(\W_{m,d})=0$ so that $\HH^1(\W_{m,d})$ is the cokernel of an injective $G$-equivariant map
   $V_{md-2} \to \HH^0(\cZ_{m,d})$. By construction, we get:
   \[
   \HH^1(\W_{m,d}) \simeq V_{md-4} \oplus \cdots \oplus V_{(m-2)d}.
   \]
   We conclude that $\FF_{\gamma,\alpha,\beta}$ is not isomorphic to $\W_{m,d}$.

   \medskip
   We turn now our attention to the dual bundle $\FF^*$ of $\FF_{\gamma,\alpha,\beta}$. 
   Let us assume that $\FF^*$ is isomorphic to some twist of $\W_{m,d}$ and prove that this forces 
   $\W_{m,d}$ to be a Tango bundle and $m=2$, $\gamma=1$, $\alpha=\beta=0$.
   Because of Lemma \ref{classi}, and since $c_1(\FF)=0$, already we should have $\NN_{m,d} \simeq \FF^*$.

   By the long exact sequence above, using again the argument of Lemma \ref{ciclico}, we see that 
   the $R$-module $\HH^2_*(\FF^*)$ is cyclic, generated in degree
   $-3\gamma$. This, together with Lemma \ref{ciclico}, implies $3\gamma=2m-1$.

   Next, the statement on the relations of the module $\HH^2_*(\W_{m,d})$ given in Lemma \ref{ciclico} 
   implies that the kernel of $R(2m-1) \to \HH^2_*(\NN_{m,d})$ has $(m-1)d+1$
   generators in degree $-m$. On the other hand, the kernel of the epimorphism $R(3\gamma) \to \HH^2_*(\FF^*)$ 
   has $d+1$ generators, the $i$-th generator being of degree $d\alpha+ i(\beta-\alpha)-2\gamma$ for $i=0,\ldots,d$.
   This implies that $m=2$ (hence $\gamma=1$) and that $d\alpha + i(\beta-\alpha) =0$ for all $i$, which taken at 
   $i=0$ and $i=d$ says $\alpha=\beta=0$, so $\FF_{\gamma,\alpha,\beta}$ is a Tango bundle.

\smallskip

   The converse implication is clear: indeed if $m=2$ we get that $\W_{2,d}^*(-1)$ is a Tango bundle on $\PP^d$ 
   and fits in the dual of the exact sequence in \cite[\S 2]{cascini:tango}.

\medskip
 Finally, let us consider the case of finite self-maps $f$ of $\PP^d$. If $f$ is defined by homogeneous polynomials of degree $e$, then the resolutions of $f^*\W_{m,d}$ and $f^*\FF_{\gamma,\alpha,\beta}^*$ are pull-back by $f$ of the
 resolutions of $\W_{m,d}$ and $\FF_{\gamma,\alpha,\beta}^*$. Note that a line bundle of the form $\OO_{\PP^d}(p)$ appearing in any of these resolutions is pulled-back by $f$ to $\OO_{\PP^d}(ep)$. So our argument excluding that 
 $\W_{m,d}$ and $\FF_{\gamma,\alpha,\beta}^*$ are isomorphic remains valid as all factors appearing in that argument get multiplied by $e$.
 Also, given our map $f$ we have:
 \[
 f_*(\OO_{\PP^d}) \simeq \bigoplus_{i=0}^s\OO_{\PP^d}(a_i),
 \]
 for some integers $s$ and $0=a_1>a_2 \ge \cdots \ge a_s$. By the projection formula, we exclude directly that 
 $\W_{m,d} \simeq \FF_{\gamma,\alpha,\beta}$ as $\HH^1_*(f^*\W_{m,d})\ne 0$ but 
 $\HH^1_*(f^*\FF_{\gamma,\alpha,\beta})=0$. 
 %\end{proof}
 
Finally we prove \ref{sasa}. Let $\mathcal{S}$ denote the Sasakura
bundle; from the monadic description given in \cite{anghel:sasakura}
we compute that the cohomology module $\HH^2_*(\mathcal{S})$ is
generated in degree $-4$. The cohomology module $\HH^2_*(\mathcal{S^*})$ of the dual bundle is
isomorphic to the previous one up to a shift, and is generated in degree $1$. Applying the same reasoning as in part
$(ii)$, we see that Lemma \ref{ciclico} implies that $2m-1=-4$ in the case of $\mathcal{S}$, and $m=1$ in the case of $\mathcal{S}^*$, and none of these two are possible. \qedhere 
\end{proof}

\bibliographystyle{amsalpha}
\bibliography{adp}

\end{document}